\pdfoutput=1
\documentclass[a4paper,11pt]{article}
\usepackage[dvipsnames,table]{xcolor}
\usepackage{dynkin-diagrams}
\usepackage{url,color,fullpage}
\usepackage{amssymb,amsmath,amsthm}
\usepackage{tocloft, paralist}
\usepackage{mdwlist}
\usepackage{mathpazo}
\usepackage{relsize}
\usepackage{mathtools}
\usepackage{comment}
\usepackage[normalem]{ulem}
\usepackage{calrsfs}
\usepackage[capitalize]{cleveref}
    \crefname{enumi}{}{}
    \Crefname{enumi}{Item}{Items}
    \crefname{equation}{}{}
    \Crefname{equation}{Equation}{Equations}

\newtheorem{theorem}{Theorem}[section]
\newtheorem{theo}[theorem]{Theorem}
\newtheorem{prop}[theorem]{Proposition}
\newtheorem{lemma}[theorem]{Lemma}
\newtheorem{cor}[theorem]{Corollary}

\newtheorem{main}[theorem]{Main Result}

\theoremstyle{definition}

\newtheorem{rem}[theorem]{Remark}

\def\<{\langle}
\def\>{\rangle}

\newcommand{\PG}{\mathsf{PG}}

\newcommand{\F}{\mathbb{F}}

\newcommand{\cP}{\mathcal{P}}

\newcommand{\cL}{\mathcal{L}}

\newcommand{\N}{\mathbb{N}}

\usepackage{xifthen}
\newcommand{\df}[2][]{\ifthenelse{\equal{#1}{}}{\index[terms]{#2}}{\index[terms]{#1}}{#2}}

\usepackage{lineno}

\begin{document}

\author{Ferdinand Ihringer\thanks{Supported by a senior postdoctoral grant of FWO-Flanders}\and Paulien~Jansen\thanks{Supported by FWO--Flanders and the L'Or\'eal-UNESCO ``For Women in Science'' program} \and Linde Lambrecht \and Yannick Neyt\thanks{Supported by the Senior Fundamental Research Project G023121N of the FWO--Flanders}  \and Daan Rijpert\thanks{Supported by Ghent University---Special Research Fund, grant BOF.24Y.2021.0029.01} \and Hendrik Van Maldeghem \and Magali Victoor}
\title{Local recognition of the point graphs of some Lie incidence geometries}
\date{}
\maketitle
\begin{abstract}
Given a finite Lie incidence geometry which is either a polar space of rank at least $3$ or a strong parapolar space of symplectic rank at least $4$ and diameter at most $4$, or the parapolar space arising from the line Grassmannian of a projective space of dimension at least $4$,  we show that its point graph is determined by its local structure. This follows from a more general result which classifies graphs whose local structure can vary over all local structures of the point graphs of the aforementioned geometries.  In particular, this characterises the strongly regular graphs arising from the line Grassmannian of a finite projective space, from the half spin geometry related to the quadric $Q^+(10,q)$ and from the exceptional group of type $\mathsf{E_6}(q)$ by their local structure. 
\end{abstract}

{MSC 2020: 51E24; 20E42. }

\section{Introduction}
Let, with standard notation, $\Gamma=(V,\sim)$ be a graph and for each $v\in X$, let $\Gamma(v)$ be the \emph{local graph at $v$}, that is, the graph induced on the neighbours of $v$. We say that a graph $\Gamma$ is locally isomorphic to a graph $\Gamma'$ if the set of local graphs $\{\Gamma(v)\mid v\in V\}$ is a subset of $\{\Gamma'(v')\mid v'\in V'\}$ (identifying isomorphism types). If all local graphs in $\Gamma'$ are isomorphic, then one wishes to conclude that each connected locally isomorphic graph is (globally) isomorphic. 
In this case, we say that $\Gamma$ is determined by its local structure. 
Call a graph $\Gamma$ \emph{locally $\Lambda$}
if all $\Gamma(v)$ are isomorphic to $\Lambda$.
For a graph $\Gamma$ locally $\Lambda$, it seems fair to say that the larger the diameter of $\Gamma$, 
the less likely it is that it is unique as being locally $\Lambda$, since in this case it 
is more likely that nontrivial quotients exist.
In particular, Weetman \cite{Wet:94} shows that if $\Lambda$ is a finite 
graph of girth at least six, then there exists an infinite graph 
$\Gamma$ which is locally $\Lambda$. 

Nevertheless, for many choices of $\Lambda$ all graphs $\Gamma$ which 
are locally $\Lambda$ have been classified. For instance, Hall shows that 
there are precisely three graphs which are locally Petersen \cite{Hal:80}.
More generally, for at least the following $\Lambda$ a classification
or at least a partial classification is known (ordered by the date of publication):
graphs which are locally polar of order $2$ \cite{Bue-Hub:77},
locally cotriangular graphs \cite{Hal-Shu:85},
locally icosahedral graphs \cite{Blo-Bro-Bus-Coh:85},
locally $\Lambda$ graphs for small $\Lambda$ \cite{Hal:85},
locally Kneser graphs \cite{Hal:87},
locally 4-by-4 grid graphs \cite{Blo-Bro:89},
locally Petersen or $K_{3,3}$-graphs \cite{Blo-Bro:92},
locally Paley graphs \cite{Bro:00},
locally co-Heawood graphs \cite{Bro-FdF-Shp:03}. 
Weetman \cite{Wet:94-2} and Brouwer \cite{Bro:95} show that $\Gamma$
has finite diameter if it is locally $\Lambda$ for many choices in which 
$\Lambda$ is strongly regular. In particular, for such $\Lambda$ there 
is a finite list of graphs for which $\Gamma$ is locally $\Lambda$.
The characterisation by local graphs is also crucial in the 
characterisation of distance-regular graphs, for instance 
see Corollary 5.6 in \cite{Gav-Koo:18} in case of the 
characterisation of certain Grassmann graphs, or also
\cite{Gav-Mak:11}. Here it is usually assumed that $\Gamma$ is distance-regular.

In the present paper, we deal with local characterisations of distance regular graphs arising from spherical buildings by taking one type of vertices of the building as the vertices of our graph, adjacent when contained in adjacent chambers. These graphs are the point graphs of the corresponding so-called \emph{Lie incidence geometries}, that is, incidence geometries arising from spherical buildings in a well-defined way. The idea is to use some local characterisation of the underlying geometries.  In order to do so, we must overcome two difficulties: (1) We must define the lines of the geometry from the given graph and the local data; (2) Since most local geometric characterisations use the framework  of the (strong) parapolar spaces, we must show that the obtained point-line geometry is a parapolar space. The case where we start with the point graph of a polar space has to be considered separately and shall be done using the axiom system of Buekenhout \& Shult \cite{Bue-Shu:74}.  

With the notation that we shall introduce in \cref{notation}, the following general local recognition theorem  is a main consequence of our results. 

\begin{main}\label{main}
Let $\Delta$ be a finite Lie incidence geometry, which is either a strong parapolar space with symplectic rank at least $4$ and diameter at most $4$, or the parapolar space arising from the line Grassmannian of a projective space of dimension at least $4$, or a polar space with rank at least $3$. Then the point graph of $\Delta$ is, as a connected graph, completely determined by its local structure.
\end{main}
This will be a consequence of the more detailed Main Results~\ref{main3} and~\ref{main2}, which we will state after introducing some preliminaries in the next section. The other sections are then devoted to the proofs of these main results. 

\cref{main} implies that the following strongly regular graphs are determined by their local structure:  the graphs on the lines of (finite) projective spaces (adjacent when non-disjoint), the point graphs of polar spaces, the graph on half of the maximal singular subspaces of the hyperbolic quadric $Q^+(10,q)$ (adjacent when intersecting in a plane) and the $\mathsf{E_{6,1}}(q)$ graph (with notation and terminology of \cite{Bro-Mal:22}). 

The case of being locally isomorphic to the line Grassmannian of a projective space of dimension at least $4$ is also covered by Corollary 5.6 of \cite{Gav-Koo:18}, although under the additional assumption that $\Gamma$ is strongly regular and has the right parameters. 

\section{Preliminaries}\label{notation}
The main players in this paper are some Lie incidence geometries arising from spherical buildings. Since these are point-line geometries, we first introduce some terminology concerning these.
\subsection{Point-line geometries}
		A \textit{point-line geometry} is a pair $\Delta=(X,\cL)$ with $X$ a set and $\cL$ a set of subsets of $\cP$. The elements of $X$ are called \textit{points}, the members of $\cL$ are called \textit{lines}. If $p \in X$ and $L \in \cL$ with $p\in L$, we say that the point $p$ \emph{lies on} the line $L$, and the line $L$ \emph{contains} the point $p$, or \emph{goes through} $p$.  If two (not necessarily distinct) points $p$ and $q$ are contained in a common line, they are called \textit{collinear}, denoted $p \perp q$ (since we will always deal with geometries in which each point lies on at least one line, we always have $p\perp p$, for each $p\in X$).  If they are not contained in a common line, we say that they are \textit{noncollinear}. For any point $p$ and any subset $P \subset \cP$, we denote \[p^\perp := \{q \in \cP\mid q \perp p\} \text{ and } P^\perp := \bigcap_{p \in P} p^\perp.\]
		A\textit{ partial linear space} is a point-line geometry in which every line contains at least three points, and where there is a unique line through every pair of distinct collinear points $p$ and $q$. That line is then denoted with $pq$. A point-line geometry is \emph{degenerate} if there is some point collinear to each point. 

		Let $\Delta = (X,\cL)$ be a partial linear space.   A subset $S\subseteq X$ is called a \textit{subspace} of $\Delta$ when every line $L$ of $\cL$ that contains at least two points of $S$, is contained in $S$. 
 A subspace $S$ in which all points are mutually collinear, or equivalently, for which $S \subseteq S^\perp$, is called a \textit{singular} subspace. If $S$ is moreover not contained in any other singular subspace, it is called a \textit{maximal} singular subspace.

We now take a look at two specific classes of point-line geometries: the polar and the parapolar spaces. 
\subsection{Polar and parapolar spaces}\label{defPPS}
Concerning polar spaces, we take the viewpoint of Buekenhout--Shult \cite{Bue-Shu:74}. 
Since it suffices in this paper to consider polar spaces of finite rank and which are not degenerate, we include this in our definition. 

A \emph{gamma space} is a point-line geometry $\Delta=(X,\cL)$ such that for each line $L\in\cL$ and each point $p\in X$, either no, or exactly one, or all points of $L$ are collinear to $p$. 

A \emph{Shult space} is a point-line geometry $\Delta=(X,\cL)$ such that for each line $\in\cL$ and each point $p\in X$, either exactly one, or all points of $L$ are collinear to $p$. 

A \emph{polar space (of rank $r$)}, $2\leq r\in\N$, is a Shult space that is not degenerate and for which the maximal singular subspaces are projective spaces of dimension $r-1$.  

Concerning parapolar spaces, we take the viewpoint of Cooperstein \cite{Coo:77}, as explained in Chapter~13 of \cite{Shu:11}. Again, it suffices to consider parapolar spaces of finite symplectic rank. The following definition is motivated by Lemma~13.4.2 of \cite{Shu:11}. 

A parapolar space of symplectic rank at least $r$ (resp.\ uniform symplectic rank $r$), $3\leq r\in\N$, is a gamma space $\Delta=(X,\cL)$ such that for each pair of distinct non-collinear points $x,y\in X$, the geometry with point set $x^\perp\cap y^\perp$ and set of lines all members of $\cL$ completely contained in $x^\perp\cap y^\perp$ is either empty, a single point, or a polar space of rank at least $r-1$ (resp.\ exactly rank $r-1$), and such that for every line $L\in\cL$, the set $L^\perp$ contains at least two non-collinear points. 

Let $\Delta=(X,\cL)$ be a parapolar space of symplectic rank at least $3$. By Lemma~13.4.1(2) of \cite{Shu:11}, the singular subspaces of $\Delta$ are projective spaces. For $x\in X$, the \emph{point residual (at $x$)} is the point-line geometry with point set the set of all lines of $\Delta$ through $x$ and with as set of lines the line pencils with vertex $x$ in some singular subspace of $\Delta$ isomorphic to a projective plane. 

Our definition of parapolar spaces  does not exclude the possibility of being a polar space. A parapolar space $\Delta=(X,\cL)$ shall be called \emph{proper} when it is not a polar space, that is, when there exist a point $p\in X$ and a line $L\in \cL$ no point of which is collinear to $p$. 

\subsection{Lie incidence geometries}

We now sketch how Lie incidence geometries arise, deferring to the literature for the precise definition of the concept of a spherical building (see for instance \cite{Abr-Bro:08,Tits:74}).  As in the latter reference, we view a spherical building as a numbered simplicial chamber complex, that is, a simplicial complex with a type function on the set of vertices such that each chamber (which is a maximal simplex) contains precisely one vertex of each type. Let $i$ be a type of a building of type $\mathsf{X}_n$, where $\mathsf{X}_n$ is a connected spherical Coxeter diagram. A simplex obtained from a chamber by deleting the vertex of type $i$ is called an \emph{$i$-panel}.  Let $X$ be the set of vertices of type $i$ and let $\cL$ be the set of subsets of $X$ with generic member the set of vertices of type $i$ forming a chamber together with a fixed $i$-panel. Then $\Delta=(X,\cL)$ is a point-line geometry, called the \emph{$i$-Grassmannian} of the corresponding spherical building, and a \emph{Lie incidence geometry of type $\mathsf{X}_{n,i}$}.

If the diagram $\mathsf{X}_n$ is simply laced and the building is finite, then the building is defined over a unique finite field $\F_q$, for some prome power $q$ and we denote the corresponding Lie incidence geometries by $\mathsf{X}_{n,i}(q)$. The Lie incidence geometries of type $\mathsf{B}_{n,1}$ are polar spaces and we will not need a special notation for them depending on the diagram.  We just remark that the polar spaces $\mathsf{D}_{n,1}$ are the point-line geometries arising from non-degenerate hyperbolic quadrics $Q^+(2n-1,q)$ in the projective space $\PG(2n-1,q)$ (using standard notation), and that the corresponding so-called half spin geometries $HS(2n-1,q)$ are the Lie incidence geometries $\mathsf{D}_{n,n}(q)$. 

A Lie incidence geometry of type $\mathsf{X}_{n,i}$ is often represented by encircling the node of type $i$ in the Coxeter diagram $\mathsf{X}_n$. This representation has the advantage of making the dimensions of the maximal singular subspaces apparent: they are equal to the lengths of the longest linear subdiagrams where the encircled node is an initial node. For instance, the maximal singular subspaces of a Lie incidence geometry of type $\mathsf{E_{6,}}$ have dimensions $4$ and $5$, as is apparent from the diagram 
\[\begin{dynkinDiagram}E{******}
\draw (root 1) circle (4pt);
\end{dynkinDiagram}\]
by deleting either the upper vertex, or the two vertices at the right. 
\subsection{Restatement of the Main Result}
Denoting the point graph of a Lie incidence geometry $\mathsf{X}_{n,i}(q)$ by $\Gamma(\mathsf{X}_{n,i}(q))$, $n\in\N$, $i\in\{1,2,\ldots,n\}$, $q$ a prime power and $\mathsf{X}_n$ a simply laced spherical Dynkin diagram, \cref{main} follows from the following more general (concerning hypotheses) and at the same time more specific (enumerating all concrete possibilities) results. 

\begin{main}\label{main3}
Let $\Gamma'$ be the disjoint union of the point graphs of all finite polar spaces of rank at least $3$. Then any connected graph $\Gamma$ locally isomorphic to $\Gamma'$ is the point graph of a finite polar space of rank at least $3$.  
\end{main}
Note that we do not even require that different connected components of $\Gamma'$ in \cref{main3} are defined over the same finite field,  nor do we assume that the graph $\Gamma$ is finite. 

For parapolar spaces, there is an explicit list \cite{Meu-Mal:22} of Lie incidence geometries that are strong parapolar spaces of symplectic rank at least $4$ and diameter at most $4$, or of symplectic rank $3$ and diameter $2$. This allows of the following statement.

\begin{main}\label{main2}
Let $\Gamma$ be a connected graph locally isomorphic to the disjoint union of  $\Gamma(\mathsf{A}_{n,2}(q))$, $n\geq 4$, $q$ ranging over all prime powers, $\Gamma(\mathsf{D}_{n,n}(q))$, $5\leq n\leq 9$, $q$ ranging over all prime powers,  $\Gamma(\mathsf{E_{6,1}}(q))$, $q$ again ranging over all prime powers, and $\Gamma(\mathsf{E_{7,7}}(q))$, $q$ once again ranging over all prime powers. Then $\Gamma$ is isomorphic to either $\Gamma(\mathsf{A}_{n,2}(q))$, $n\geq 4$, $\Gamma(\mathsf{D}_{n,n}(q))$, $5\leq n\leq 9$, $\Gamma(\mathsf{E_{6,1}}(q))$, or $\Gamma(\mathsf{E_{7,7}}(q))$, for some prime power $q$. 
\end{main}

Note that, for any prime power $q$, $\Gamma(\mathsf{A}_{n,2}(q))$, $n\geq 4$, $\Gamma(\mathsf{D_{5,5}}(q))$ and  $\Gamma(\mathsf{E_{6,1}}(q))$ are strongly regular (hence have diamter 2), whereas  $\Gamma(\mathsf{D_{6,6}}(q))$, $\Gamma(\mathsf{D_{7,7}}(q))$ and  $\Gamma(\mathsf{E_{7,7}}(q))$ have diameter $3$, and $\Gamma(\mathsf{D_{8,8}}(q))$ and $\Gamma(\mathsf{D_{9,9}}(q))$ have diameter $4$. 

Since the vertices of the graphs we will consider are the points of a Lie incidence geometry, we will from now on deviate from standard notation and denote the vertex set of a graph by $X$.

\section{Proof of \cref{main3} and most of \cref{main2}}
Le $\Gamma=(X,\sim)$ be a graph and $q$ a natural number. The \emph{$q$-clique extension $q\Gamma$} of $\Gamma$ is the graph with vertices $t_i(x)$, $i\in\{1,2,\ldots,q\}$, $x\in X$, with $t_i(x)$ and $t_j(y)$ adjacent if either $x=y$ and $i\neq j$, or $x\sim y$.  If the set of vertices equal or adjacent to a vertex $x\in X$ coincides with the set of vertices equal or adjacent to $y\in X$,  then the sets $\{t_i(x)\mid i=1,2,\ldots,q\}$ and $\{t_i(y)\mid i=1,2,\ldots,q\}$ cannot be distinguished in $q\Gamma$ (in fact, in their union, every vertex plays the same role). However, this is the only obstruction, as we will show below. For an arbitrary set $S$ of vertices, we denote by $S^\perp$ the set of vertices equal of adjacent to every vertex in $S$. For $S=\{x\}$, we denote $S^\perp=x^\perp$.

\begin{lemma}\label{lem1}
Let $\Gamma=(X,\sim)$ be a graph and let $q$ be a natural number. Suppose $x\in X$ has the property that the set $(x^\perp)^\perp$ coincides with $\{x\}$. Then in $q\Gamma$, we have $\{t_i(x)\mid i=1,2,\ldots,q\}=(t_j(x)^\perp)^\perp$, for each $j\in\{1,2\ldots,q\}$. 
\end{lemma} 

\begin{proof}
This follows immediately from the definition of $q\Gamma$. 
\end{proof}

Under the assumptions of \cref{lem1}, we call the set $\{t_i(x)\mid i=1,2,\ldots,q\}$ a \emph{ray}, or \emph{the ray of $x$}, and we say that the ray is \emph{reconstructable}. Also, the natural number $q$ is called the \emph{height} of the graph and is well defined under the assumptions of \cref{lem1}.

Since each vertex of the point graph of a polar space has the property mentioned in \cref{lem1}, each point of the point graph of any Lie incidence geometry which is a (proper) parapolar space also has that property. 

Now 
let $\Gamma=(X,\sim)$ be a connected graph which has at each of its vertices the local structure of the $q$-clique extension of the point graph of a Lie incidence geometry which is either a  parapolar space of symplectic rank at least $3$ and diameter $2$ in which the lines carry $q+1$ points, or a polar space of rank at least $2$ in which the lines carry exactly $q+1$ points (also $q$ depends on the vertex). This Lie incidence geometry is called the \emph{local geometry at the corresponding vertex} and denoted $\Delta(x)$; hence $\Gamma(x)\cong q\Delta(x)$.  Let $p\in X$ be an arbitrary vertex of $\Gamma$ and let $x\sim p$; so $x\in\Gamma(p)$.  Then, by our observation, the rays of the local graph $\Gamma(p)$ are reconstructable. The ray $R_x$ to which $x$ belongs is equal to $(x^{\perp_p})^{\perp_p}$, where $\perp_p$ denotes adjacency in $\Gamma(p)$. Hence $x^{\perp_p}=(x^\perp\cap p^\perp)\setminus\{p\}$ and so $R_x=(((x^\perp\cap p^\perp)\setminus\{p\})^\perp\cap p^\perp)\setminus\{p\}$. Then \[R:=R_x\cup\{p\}=((x^\perp\cap p^\perp)\setminus\{p\})^\perp\cap p^\perp=(x^\perp\cap p^\perp)^\perp.\] 
Now note that the latter is symmetric in $x$ and $p$, hence the ray  in $\Gamma(x)$ to which $p$ belongs is equal to $R\setminus\{x\}$. Since $R_x$ is determined in $\Gamma(p)$ by any of its members $y$, we now deduce that $R$ is determined by any pair $(p_1,p_2)$ of its points as $p_1\cup R_{p_2}$, with $R_{p_2}$ the ray in $\Gamma(p_1)$ to which $p_2$ belongs. We denote the set $R$ by $R[p_1,p_2]$ and call it an \emph{extended ray}. This already has, by connectivity,  the following consequence. 

\begin{lemma}\label{heights}
The heights of the local graphs at two distinct vertices of $\Gamma$ coincide. 
\end{lemma}

We now define the set $\cL$ as the set of all extended rays $R[p_1,p_2]$, with $p_1\sim p_2$, for $p_1,p_2\in X$ and we define the geometry $\Delta=(X,\cL)$. Clearly, the point graph of $\Delta$ is $\Gamma$. Also, the fact that lines are determined by any pair of their points translates in the property that $\Delta$ is a partial linear space. 

\begin{lemma}\label{gamma}
The geometry $\Delta=(X,\cL)$ is a gamma space. 
\end{lemma}

\begin{proof}
Let $p,x,y\in X$, with $p\sim x\sim y\sim p$ and $R[p,x]\neq R[p,y]$. Note that the rays in $\Gamma(p)$ correspond to the points of $\Delta(p)$. Since in $\Delta(p)$ the line through two collinear points $u,v$ is given by $(u^\perp\cap v^\perp)^\perp$ (with $\perp$ the usual collinearity relation including equality),  in $\Gamma(p)$, the union $U$ of the rays corresponding to the line of  $\Delta(p)$ through the points $R[p,x]\setminus\{p\}$ and $R[p,y]\setminus\{p\}$  is given by \[U=(((x^\perp\cap y^\perp\cap p^\perp)\setminus\{p\})^\perp\cap p^\perp)\setminus\{p\},\]
which, as above, equals \[(x^\perp\cap y^\perp\cap p^\perp)^\perp\setminus\{p\}.\] Since $R[x,y]$ does not contain $p$ and since \[(x^\perp\cap y^\perp)^\perp\subseteq (x^\perp\cap y^\perp\cap p^\perp)^\perp,\] we see that $R[x,y]\subseteq U$. Now $R[x,y]$ has at most one vertex in common with each ray in $\Gamma(p)$ (as extended rays are determined by two points and $R[x,y]$ does not contain $p$).  Since $|R[x,y]|=q+1$ and there are precisely $q+1$ rays of $\Gamma(p)$ in $U$ (as there are $q+1$ points on each line in $\Delta(p)$),   we conclude that $p$ is collinear to all points of $R[x,y]$ in $\Delta$. 
\end{proof}

The proof of \cref{gamma} yields the following consequence. 

\begin{cor}\label{projplane}
Let $x,y,z\in X$ be pairwise adjacent with $x\notin R[y,z]$. Then $\pi:=(x^\perp\cap y^\perp\cap z^\perp)^\perp$ is, endowed with all members of $\cL$ contained in it, a projective plane.  In each of $\Gamma(x), \Gamma(y),\Gamma(z)$, the point set $\pi$ minus $x,y,z$, respectively, represents a line in the corresponding local geometry. 
\end{cor}
\begin{proof}
Just like extended rays are determined by any pair of its points, one shows that $\pi$ is determined by any triple of its points not contained in a single extended ray. Hence we may think of two arbitrary extended rays contained in $\pi$ as, with the notation of the proof of \cref{gamma}, one containing $p$ and the other containing $x$. Then the proof of \cref{gamma} shows that these two rays intersect in a unique point.  

The last assertion then also follows from thinking of $x,y$ and $z$ as the vertex $p$ in the proof of \cref{gamma}. 
\end{proof}
An immediate consequence of \cref{projplane} is the following.

\begin{cor}\label{cliques}
Let $C$ be a maximal clique of $\Gamma$ and $v\in C$. Then $C$, endowed with the extended rays, is a projective space, say of dimension $k$ which corresponds in the local geometry at $v$ to a maximal singular subspace of dimension $k-1$. 
\end{cor}

It follows now by connectivity that the maximal singular subspaces of all local geometries have the same dimension (since all local geometries we consider admit a point transitive automorphism group). Hence, in order to show \cref{main3}, it suffices to show the following proposition.

\begin{prop}\label{polarspace}
If for each vertex $p\in X$ the local geometry $\Delta(p)$ is a polar space of rank $r\geq 2$ having $q+1$ points per line, then $\Delta$ is a polar space of rank $r+1$ and, consequently, $\Gamma$ being the point graph of $\Delta$, it is the point graph of a polar space.  
\end{prop}

\begin{proof}
We begin with showing that $\Delta$ is a Shult space. Since $\Delta$ is a gamma space, we only have to prove that each line $L$ has at least one point collinear with each point $p$. This is trivial if $p\in L$, so assume $p\notin L$. Without loss of generality we may assume for a contradiction, and by connectivity, that there is no vertex on $L$ adjacent to $p$ in $\Gamma$, but there exists a vertex $y\in X$ adjacent to $p$ and adjacent to some point $x\in L$. Since $\Delta(x)$ is a polar space of rank at least $2$, we find a point $z\in \Gamma(x)\cap\Gamma(y)$ not on $R[x,y]$. In $\Delta(y)$, the extended ray $R[y,p]$ represents a point $p^*$ and, by \cref{projplane}, the set $\pi:=(x^\perp\cap y^\perp\cap z^\perp)^\perp$ represents a line $L^*$. Then there is a point $v^*$ on $L^*$ collinear to $p^*$ in $\Delta(y)$. This translates in $\Gamma(y)$ to the existence of some vertex $v\in \pi\setminus \{y\}$ adjacent to $p$. Since $v\sim p$, all vertices of $R[y,v]$ are adjacent to all those of $R[y,p]$. Again by \cref{projplane}, the sets $R[y,v]$ and $R[x,z]$ intersect in some point $s\sim p$. 

If $y$ is adjacent to all vertices of $L$, then we could have chosen $z$ on $L$ and $p\sim s\in L$. 

If not, then we can choose $z$ collinear to all points of $L$; it follows that $s$ is also collinear to all points of $L$. Letting $s$ play the role of $y$, we are now back to the situation in the previous paragraph and find a point on $L$ adjacent to $p$. This shows that $\Delta$ is a Shult space. 

It remains to show that no vertex is adjacent to all other vertices. If some vertex $v$ were adjacent to all other vertices, then clearly, for each $w\in X\setminus\{v\}$, the local geometry $\Delta(w)$  would be degenerate.  
\end{proof}

The parapolar spaces $\mathsf{A_{1,1}}(q)\times \mathsf{A}_{n,1}(q)$, $n\geq 2$, have maximal singular subspaces which are lines. This is not true in any other parapolar space which is isomorphic to a point residual of one of the parapolar spaces mentioned in the hypotheses of \cref{main2}. So we may assume that either all local geometries are isomorphic to $\mathsf{A_{1,1}}(q)\times \mathsf{A}_{n,1}(q)$, $n\geq 2$, or none are. In the present section, we continue with the latter assumption, delaying the proof of the former to the next section. 

\begin{prop}\label{perpquad}
If for each vertex $p\in X$ the local geometry $\Delta(p)$ is a strong parapolar space of uniform symplectic rank $r\geq 3$ having $q+1$ points per line, then for each pair of points $x,y\in X$ of $\Delta$ at distance $2$ in $\Gamma$, the subgeometry $x^\perp\cap y^\perp$ of $\Delta$ is a polar space of rank $r$.
\end{prop}

\begin{proof}
Let $x,y\in X$ be two vertices of $\Gamma$ at mutual distance $2$ and let $p\in x^\perp\cap y^\perp$ be arbitrary. Let $x^*$ and $y^*$ be the points of $\Delta(p)$ corresponding to the extended rays $R[p,x]$ and $R[p,y]$, respectively. Since $\Delta(p)$ is a strong parapolar space of diameter $2$ and symplectic rank $r\geq 3$, the set ${x^*}^\perp\cap {y^*}^\perp$ defines a polar space $\Delta'(p)$ of rank $r-1$.  It follows that the local structure of the graph $\Gamma(x)\cap\Gamma(y)$ at $p$ is the $q$-clique extension of $\Delta'(p)$. Since this holds for every $p\in\Gamma(x)\cap\Gamma(y)$, \cref{polarspace} implies that $\Gamma(x)\cap\Gamma(y)$ is the point graph of a polar space of rank $r$. Since the extended rays are the lines of that polar space, the geometry  $x^\perp\cap y^\perp$ endowed with the extended rays is a polar space of rank $r$. 
\end{proof}

\begin{rem}
In the previous proposition, we may weaken the assumptions to the local geometries having symplectic rank \emph{at least $r\geq 3$} (with the same proof).  However, since in all our applications, the rank is constant, we limit ourselves to this case. The same remark applies to the next proposition.
\end{rem}

\begin{prop}\label{paraspace}
If for each vertex $p\in X$ the local geometry $\Delta(p)$ is always a strong parapolar space of uniform symplectic rank $r\geq 3$ having $q+1$ points per line, then the geometry $\Delta$ is a parapolar space of symplectic rank $r+1$. Also, the local geometry in $\Gamma$ at the point $p\in X$ is precisely the point residual at $p$ in $\Delta$. 
\end{prop}

\begin{proof}
By the definition of parapolar spaces given in \cref{defPPS},  we have to show that 
\begin{compactenum}[$(i)$]
\item \emph{$\Delta$ is a connected gamma space}. This is true by \cref{gamma} (and the assumption that $\Gamma$ is connected);
\item \emph{for each pair of distinct non-collinear points $x,y$, the geometry induced on $x^\perp\cap y^\perp$ is either empty, a single point or a polar space.} This is true by \cref{perpquad}.
\item \emph{For each line $L$, the set $L^\perp$ contains a pair of non-collinear points.} This follows from the well-definedness of the extended ray $R[x,y]$ in $\Gamma(x)$, where $x,y\in L$, $x\neq y$.
\end{compactenum}
The last assertion is immediate. This completes the proof. 
\end{proof}

We are now ready to prove part of \cref{main2}. We reformulate.

\begin{theo}\label{withoutsegre}
Let $\Gamma$ be a connected graph locally isomorphic to the disjoint union of  $\Gamma(\mathsf{D}_{n,n}(q))$, $4\leq n\leq 9$, $q$ ranging over all prime powers,  $\Gamma(\mathsf{E_{6,1}}(q))$, $q$ again ranging over all prime powers, and $\Gamma(\mathsf{E_{7,7}}(q))$, $q$ once again ranging over all prime powers. Then all local geometries of $\Gamma$ are mutually isomorphic and $\Gamma$ is isomorphic to either $\Gamma(\mathsf{D}_{n,n}(q))$, $\Gamma(\mathsf{E_{6,1}}(q))$, or $\Gamma(\mathsf{E_{7,7}}(q))$, for some prime power $q$. 
\end{theo}
\begin{proof}
We know by \cref{paraspace} that $\Delta=(X,\cL)$ is a parapolar space of symplectic rank at least $4$. Now consider two adjacent vertices $x,y\in X$ of $\Gamma$. Then, by \cref{cliques}, the dimensions of the maximal singular subspaces in the local geometries $\Delta(x)$ and $\Delta(y)$, through the points corresponding to the extended ray $R[x,y]$, are the same. However, for the given parapolar space, these (well-known) dimensions (for each point) are the following:
\renewcommand{\arraystretch}{1.6}
\[\begin{array}{l|l|l}
\mathsf{A_{4,2}}(q) & \begin{dynkinDiagram}A{*****}
\draw (root 2) circle (4pt);
\end{dynkinDiagram} & 2\mbox{ and }3\\ 
\mathsf{A_{5,2}}(q) & \begin{dynkinDiagram}A{******}
\draw (root 2) circle (4pt);
\end{dynkinDiagram} & 2\mbox{ and }4\\ 
\mathsf{A_{6,2}}(q) & \begin{dynkinDiagram}A{*******}
\draw (root 2) circle (4pt);
\end{dynkinDiagram} & 2\mbox{ and }5\\ 
\mathsf{A_{7,2}}(q) & \begin{dynkinDiagram}A{********}
\draw (root 2) circle (4pt);
\end{dynkinDiagram} & 2\mbox{ and }6\\ 
\mathsf{A_{8,2}}(q) & \begin{dynkinDiagram}A{*********}
\draw (root 2) circle (4pt);
\end{dynkinDiagram} & 2\mbox{ and }7\\ 
\mathsf{D_{5,5}}(q) & \begin{dynkinDiagram}D{*****}
\draw (root 4) circle (4pt);
\end{dynkinDiagram} & 3\mbox{ and }4\\
\mathsf{E_{6,1}}(q) & \begin{dynkinDiagram}E{******}
\draw (root 1) circle (4pt);
\end{dynkinDiagram} & 4\mbox{ and }5
\end{array}  \]
Hence all local geometries are isomorphic. If these local geometries, which are the point residuals, are of type $\mathsf{A}_{n,2}$, $4\leq n\leq 8$, then by  Lemma~4.6 of \cite{Coh-Sch-Sch-Mal:22} (see also Lemma~5.3 of \cite{ForumSigma}), $\Delta$ is isomorphic to $\mathsf{D}_{n+1,n+1}(q)$. If these local geometries are isomorphic to $\mathsf{D_{5,5}}(q)$, then by Lemma~5.1 of \cite{ForumSigma}, $\Delta$ is isomorphic to $\mathsf{E_{6,1}}(q)$. Finally, if these local geometries are isomorphic to $\mathsf{E_{6,1}}(q)$, then by Lemma~5.5 of \cite{ForumSigma}, $\Delta$ is isomorphic to $\mathsf{E_{7,7}}(q)$. 

Since the point graph of $\Delta$ is $\Gamma$, the proof is complete.   
\end{proof}

\section{The case of symplectic rank 2 for the local geometry}

In this section, we tackle the remaining case of \cref{main2}: We assume all local geometries of $\Gamma$ are parapolar spaces isomorphic to $\mathsf{A_{1,1}}(q)\times \mathsf{A}_{n,1}(q)$, $n\geq 2$, for some (non-constant) prime power $q$ and some (non-constant) natural number $n$. The same arguments as in the previous section show that $q$ and $n$ are in fact constants. We define $\Delta$ in the same way as before, and we first show that in $\Delta$, for every pair of  points $x,y$ at mutual distance $2$, the geometry induced on $x^\perp\cap y^\perp$ by the extended rays is a generalised quadrangle isomorphic to a $(q+1)\times(q+1)$-grid. Where we previously could use \cref{polarspace} to prove that this geometry is a polar space, this fails now as it has rank 2. However, we propose an alternative argument in this specific case.  

\begin{lemma}
For every pair of points $x,y$ of $\Delta$ at mutual distance $2$, the geometry induced on $x^\perp\cap y^\perp$ by the extended rays is a $(q+1)\times(q+1)$-grid, where $q$ is the size of any ray.
\end{lemma} 
\begin{proof}
Let $x\sim p\sim y$.  In $\Gamma(p)$ we find vertices $v$ and $w$ such that $R[p,x]\sim R[p,v]\sim R[p,y]\sim R[p,w]\sim R[p,x]$. Moreover, we may assume that the planes $\alpha_1:=(x^\perp\cap p^\perp\cap v^\perp)^\perp$ and $\beta_1=(y^\perp\cap p^\perp\cap w^\perp)^\perp$ correspond to maximal singular subspaces (hence lines) of $\Delta(p)$. 

In $\Gamma(x)$, the extended rays $R[x,w]$ and $R[x,v]$, which correspond to points of $\Delta(x)$ at mutual distance $2$, are adjacent to a unique common ray $R[x,r]$. Note that the set $U_1:=x^\perp\cap r^\perp\cap v^\perp$, endowed with the extended rays, is an $(n+1)$-dimensional projective space. 

Likewise, in $\Gamma(y)$, there is a vertex $s$ such that $R[y,v]\sim R[y,s]\sim R[y,w]$. The set $\beta_2:=y^\perp\cap v^\perp\cap s^\perp$, endowed with the extended rays, is a projective plane, and $V_1=y^\perp\cap s^\perp\cap w^\perp$ defines an $(n+1)$-dimensional singular subspace. 

Now, in $\Delta(v)$, two maximal singular subspaces of distinct dimension intersect in a unique point, hence $U_1\cap\beta_2=R[v,t]$, for some $t\in X$. Since $\beta_2$ is a plane, we can assume $t\in R[y,s]$. Then $t\in V_1$ and so $t\sim w$. Hence in $x^\perp\cap y^\perp$ we find the quadrangle $p\sim v\sim t\sim w\sim p$. 

Let $z$ be any vertex on $R[t,v]$. Then there is a unique $(n+1)$-dimensional singular subspace $V_2$ of $\Delta$ through $R[y,z]$, and it intersects $\beta_1$ in a unique extended ray $R[y,z']$. Clearly we can choose $z'\in R[p,w]$ (since $\beta_2$ is a projective plane). Likewise, for each vertex $u\in R[t,w]$, the unique $(n+1)$-dimensional singular subspace through $R[x,u]$ intersects the plane $\alpha_1$ in a ray $R[x,u']$, with $u'\in R[p,v]$. Moreover, $\alpha_3:=x^\perp\cap z^\perp\cap {z'}^\perp$ defines a plane and $U_2:=x^\perp\cap u^\perp\cap {u'}^\perp$ defines an $(n+1)$-dimensional singular subspace. In $\Delta(x)$, these spaces intersect in a unique point, which implies that the extended rays $R[z,z']$ and $R[u,u']$ intersect in a point. Hence, varying $z$ on $R[t,v]$ and $u$ on $R[t,w]$, we obtain a $(q+1)\times(q+1)$-grid in $x^\perp\cap y^\perp$.

It remains to show that there are no other vertices contained in $x^\perp\cap y^\perp$. Suppose for a contradiction that some vertex $a\in X$ not on the above grid is adjacent to both $x$ and $y$. Since $R[p,v]$ intersects every maximal singular subspace of dimension $n+1$ of $\Delta$ through $x$, we may assume that $a\in  U_1\setminus R[v,t]$. Then the plane $(a^\perp\cap t^\perp\cap v^\perp)^\perp$ is contained in the two distinct maximal cliques $a^\perp\cap t^\perp\cap v^\perp\cap x^\perp$ and   $a^\perp\cap t^\perp\cap v^\perp\cap y^\perp$. This contradicts the local structure in $\Gamma(a)$. 
\end{proof}

Now the proof of \cref{paraspace} can be repeated verbatim, and we have the following extension of \cref{withoutsegre}.

\begin{theo}\label{withsegre}
Let $\Gamma$ be a connected graph locally isomorphic to the disjoint union of  $\Gamma(\mathsf{A}_{n,2}(q))$, $n\geq 4$, $q$ ranging over all prime powers, $\Gamma(\mathsf{D}_{n,n}(q))$, $4\leq n\leq 9$, $q$ ranging over all prime powers,  $\Gamma(\mathsf{E_{6,1}}(q))$, $q$ again ranging over all prime powers, and $\Gamma(\mathsf{E_{7,7}}(q))$, $q$ once again ranging over all prime powers. Then all local geometries of $\Gamma$ are mutually isomorphic and $\Gamma$ is isomorphic to either $\Gamma(\mathsf{A}_{n,2}(q))$, $n\geq 4$, $\Gamma(\mathsf{D}_{n,n}(q))$, $4\leq n\leq 9$, $\Gamma(\mathsf{E_{6,1}}(q))$, or $\Gamma(\mathsf{E_{7,7}}(q))$, for some prime power $q$. 
\end{theo}
\begin{proof}
In view of \cref{withoutsegre}, we may assume that at some vertex the local geometry is  $\mathsf{A_{1,1}}(q)\times \mathsf{A}_{n,1}(q)$, for some $n\geq 2$ and some prime power $q$. \cref{heights} implies that the local geometry at each vertex is defined over $\F_q$. Moreover, since the maximal singular subspaces of $\mathsf{A_{1,1}}(q)\times \mathsf{A}_{n,1}(q)$ have dimensions $1$ and $n$, the argument in the proof of \cref{withoutsegre} using \cref{cliques} shows that the local geometry at each vertex is $\mathsf{A_{1,1}}(q)\times \mathsf{A}_{n,1}(q)$. 

Hence all point residuals of the parapolar space $\Delta$ are isomorphic to $\mathsf{A_{1,1}}(q)\times \mathsf{A}_{n,1}(q)$. Then Lemma 5.4 of \cite{ForumSigma} implies that $\Delta$ is isomorphic to $\mathsf{A}_{n+2,2}(q)$. Since the point graph of $\Delta$ is $\Gamma$, the proof is complete.   
\end{proof}

\begin{rem}
One can merge Main Results~\ref{main3} and~\ref{main2} by assuming that each residual geometry is either a polar space of rank at least $3$, or a parapolar space as in \cref{main2}. The conclusion is then that $\Gamma$ is the point graph of either a polar space, or one of the geometries in the conclusion of \cref{main2}. The reason is that polar spaces only have one type of maximal singular subspaces, and then with the help of \cref{cliques}, one concludes that in all points the residuals are either polar spaces---and we are in the case of \cref{main3}--- or proper parapolar spaces---and we are in the case of \cref{main2}.
\end{rem}

\begin{rem}
In \cref{main2} we can slightly relax the restriction on the diameter of  $\Gamma(\mathsf{D}_{n,n}(q))$ to also include the case $n=10$. Indeed, in this case the proof \cref{withoutsegre} implies that $\Gamma$ is either $\Gamma(\mathsf{D}_{10,10}(q))$, or a proper quotient of $\Gamma(\mathsf{D}_{10,10}(q))$ with respect to a nontrivial automorphism group $G$. However, suppose we are in the latter case and let $g\in G$ be nontrivial. In order to preserve the local structure, $g$ must map every vertex to a vertex at distance $5$ (see Lemma~A.5 of \cite{ForumSigma}), which is an opposite vertex in the corresponding building. This contradicts Theorem~1.2 of \cite{Dev-Par-Mal:13}. Hence no proper quotients occur. 

Without proof we mention that more elaborate arguments also prove that  $\Gamma(\mathsf{D}_{n,n}(q))$, $11\leq n\leq 17$, does not admit a quotient with the same local structure.  
\end{rem}

{\it Affiliations Ferdinand Ihringer:}\\
Dept.~of Mathematics:~Analysis, Logic and Discrete Math., Ghent University, Belgium.\\
Dept.~of Mathematics, Southern University of Science and Technology, Shenzhen, China.\\
E-mail: {\tt ferdinand.ihringer@gmail.com}.

{\it Affiliation other authors:}\\
Dept.~of Mathematics:~Algebra and Geometry, Ghent University, Belgium.\\
Email addresses: \\
{\tt Paulien.Jansen@UGent.be\\
Linde.Lambrecht@UGent.be\\
Yannick.Neyt@UGent.be\\
Daan.Rijpert@UGent.be\\
Hendrik.VanMaldeghem@UGent.be\\
Magali.Victoor@UGent.be}
\end{document}